\documentclass[12pt]{article}
\usepackage{amsmath,amsfonts,amsthm,amssymb,blindtext
,graphicx,pifont,fncylab,theoremref,sectsty,tocloft,titlesec,enumitem,
fmtcount,titletoc}
\date{}
\newtheorem{theorem}{Theorem}

\newtheorem{prp}{Proposition}

\begin{document}

\baselineskip 20pt
\title{On the boundary behavior of analytic functions in the unit disc}

\author{Spyros Pasias}

\maketitle
\begin{center}

\indent Orcid:0000-0001-6611-0155\\
\indent pasias.spyros@ucy.ac.cy\\
\indent University of Cyprus
\end{center}
\begin{abstract}
In this paper we will deal with problems in approximation theory of bounded analytic functions on the unit disc and their boundary behavior on the unit circle. We will attempt to unify two known such theorems one of which is a classical theorem of S.V Kolesnikov \cite{KO}, to create a stronger theorem. We will solve various special cases of the unification problem, and in particular one case extends the main result found in \cite{PE} regarding the boundary behavior of Blaschke products, meanwhile also providing a more compact and elegant approach for the proof of its necessity. The necessity part of the proof uses a classical theorem of Baire. Lastly, we will prove an analogue of the classical theorem by Kolesnikov for Blaschke products provided an extra requirement is satisfied and use the main result in \cite{PE} to simplify the proof of an important result found in \cite{PI}.
\end{abstract}

\section{Introduction}

Approximation theory plays a crucial role in the study of bounded analytic functions, particularly in understanding their behavior within the unit disc $D$ and on its boundary, the unit circle $T$. Over the years, various theorems have addressed the boundary behavior of such functions, highlighting both their theoretical significance and practical applications in complex analysis. 

In this work inspired by my doctorate thesis\cite{Spy}, we aim to build upon this established framework by seeking a unification of two significant theorems in the field---one of which is Kolesnikov’s classical result. Our objective is not only to develop a stronger unified theorem but also to address particular cases that extend recent results concerning the boundary behavior of \emph{Blaschke products}. Recall that a Blaschke product is an infinite product of the form
\[
B(z,A) = \prod_{n=1}^{\infty} \frac{|a_n|}{a_n} \cdot \frac{a_n - z}{1 - \overline{a_n}\, z},
\]
where \(A=\lbrace{ a_n\rbrace}_{n=1}^\infty\) is a sequence in the open unit disk $D$, that approaches the boundary unit circle $T$ fast enough. Precisely, satisfying
\[
\sum_{n=1}^{\infty} \bigl(1 - |a_n|\bigr) < \infty.
\]
These functions are bounded analytic functions on the unit disk whose zero sets coincide exactly with the sequence \((a_n)\). Blaschke products play a fundamental role in the Factorization Theorem for bounded analytic functions and in the construction of explicit examples with prescribed zeros and boundary behavior.\\
Our investigation also simplifies key proofs, offering more elegant approaches to known problems, while leveraging tools such as Baire’s classical theorem. In addition, we explore an analogue of Kolesnikov’s theorem under specific conditions, demonstrating its relevance to recent findings and applications.

By approaching these problems with both classical and modern techniques, this paper contributes to a deeper understanding of approximation theory and its implications for bounded analytic functions, opening pathways for further exploration in complex analysis.

\section{Important theorems and the unification problem}
\noindent We start by stating some important theorems about the boundary behavior of bounded analytic functions. Throughout the paper we let $D$ represent the unit disc and $T$ the unit circle. The following important theorem is due to S.V Kolesnikov \cite{KO}.

\begin{theorem}[Kolesnikov's Theorem] Let $E\subset T$ be of type $G_{\delta\sigma}$ and of measure $0$. Then there exists a bounded analytic function in the unit disc $D$ which fails to have radial limits exactly at the points of $E$.
\end{theorem}

\noindent It should be noted that the converse statement of {Kolesnikov's Theorem} which states:\\
\noindent The set in which the radial limits of a bounded analytic function is necessarily $G_{\delta\sigma}$, is elementary and has been known long before {Kolesnikov's Theorem} was introduced. A proof can be found in (\cite{AL}, p. 23).

\begin{theorem}[Fatou's interpolation theorem, 1906] Let $K$ be a closed subset of $T$ such that $m(K)=0$. Then there exists a function $f$ in the disc algebra which vanishes precisely on $K$.
\end{theorem}

\noindent A proof of {Fatou's interpolation theorem}, (that shows the real part of $f$ is positive on $D$) can be found in (\cite{KO}, p. $29-30$). A new proof may also be found in \cite{ART}. The construction of Fatou's function is widely used and has many applications. In particular it is used to prove an important result by \emph{F. and M. Riesz} which states that any analytic measure is absolutely continuous with respect to the Lebesgue measure.\\
\noindent The necessity part of the following theorem is well known and its proof is elementary. The sufficiency part is proven in \cite{Da} and its proof is using the method of the paper \cite{AD} and is based on {Fatou's interpolation theorem}.

\begin{theorem}
Let $E$ be a set on $T$. Then there exists a function $f\in H^\infty$ which has no radial limits on $E$ but has unrestricted limits at each point of $T\setminus E$ if and only if $E$ is an $F_\sigma$ set of measure zero.
\end{theorem}

\noindent The following theorem is a solution to a problem proposed by Rubel in 1973 and was solved in \cite{DA} with an affirmative answer to the problem. The problem asks whether for any $G_\delta$ set $F\subset T$ of measure zero, there exists a non vanishing bounded analytic function $f$ on $D$ such that $f=0$ precisely on $F$ and such that the radial limits of $f$ exist everywhere on $T$. The following stronger theorem found in \cite{art} provides an affirmative answer to Rubel's problem as well.

\begin{theorem} Let $F$ be a $G_\delta$ of measure zero on $T$. Then there exists a nonvanishing bounded analytic function $g$ such that:\begin{enumerate}
 \item $g$ has non zero radial limits everywhere on $T\setminus F$.
 \item $g$ has vanishing unrestricted limits at each point of $F$.
 \end{enumerate}
\end{theorem}

\noindent Now we use the three theorems above to solve a particular case of Problem $1$ stated in \cite{Da}. Although the problem stated in \cite{Da} is still an open problem and a good research opportunity, we have solved several particular cases of it in this paper. The problem is an attempt to unify both {Kolesnikov's Theorem} and {Theorem 3} to a single stronger theorem. The problem we are interested in (See [11], p. 4) is the following:\vspace{0.12in}

\noindent\textbf{Unification Problem.}
\emph{Let $E_1\subset E_2$ be subsets of the unit circle $T\equiv\lbrace z\mid |z|=1 \rbrace$. Find necessary and sufficient conditions for the existence of an $f\in H^\infty$ such that the radial limits of $f$ fail to exist precisely on $E_1$ and it has unrestricted limits precisely at $T\setminus E_2$.}\\

\noindent Obvious necessary conditions are that $E_1$ is of type $G_{\delta\sigma}$ and $E_2$ is of type $F_\sigma$. The case $E_1=\emptyset$ is also dealt with by the following theorem.

\begin{theorem}[Brown, Gauthier and Hengartner]
Let $E$ be a set on $T$. Then there exists a function $f\in H^\infty$ which has no unrestricted limits on $E$ but has unrestricted limits at every point of $T\setminus E$ if and only if $E$ is of type $F_\sigma$.
\end{theorem}

\noindent The proof of the theorem above is not difficult. The idea is to express $E$ as a countable union of closed sets $E=\bigcup_n C_n$, and then for each of the sets $C_n$ pick a respective Blaschke product $B_n$ with zeros accumulating precisely on $C_n$. Then, since Blaschke products are bounded by modulus by $1$, the infinite sum $\sum_{n=1}^\infty \frac{B_n}{2^n}$ converges to the function with the desired properties.\vspace{0.12in}

\noindent Now we will proceed and solve some particular cases of the Unification Problem. In the proofs that follow, we note that each function of the unit disc is extended to take the value of its radial limits on the unit circle $T$ when they exist.
\section{A case of the unification Theorem that extends Collewl's Theorem}
The following theorem is an elementary case of the unification problem.\\

\begin{theorem} Let $E_1\subset E_2$ be subsets of the unit circle $T$ and suppose that the measure of $E_2$ is $2\pi$. The following conditions are necessary and sufficient for the existence of an $f\in H^\infty$ such that the radial limits of $f$ fail to exist precisely at $E_1$ and it has vanishing unrestricted limits precisely at $T\setminus E_2$:
\begin{enumerate}
\item $E_1$ is of type $G_{\delta\sigma}$ and of measure zero.
\item $E_2$ is of type $F_\sigma$.
\end{enumerate}
\end{theorem}

\begin{proof}
\textbf{(Necessity)} It is well known that the exceptional set of points in $T$ such that $f$ fails to have radial
 limits is of type $G_{\delta\sigma}$ and by {Fatou's Theorem} is of measure zero. 
 Also if $f$ has unrestricted limits on $T\setminus E_2$ then $E_2=\bigcup_{n=1}^{\infty} F_n$ where $F_n$ are the closed sets consisting 
 of all points in $E_2$ any neighborhood of which contains points $t_1,t_2\in D$ such that $|f(t_1)-f(t_2)|>\frac{1}{n}$. Since clearly each $F_n$ is closed it follows that $E_2$ is of type $F_{\sigma}$ and the proof of necessity is complete.\\
\textbf{(Sufficiency)} Since $E_1$ is of type $G_{\delta\sigma}$ and of measure zero by {Kolesnikov's Theorem} there exists a bounded analytic function $f_1$ that fails to have radial limits exactly on $E_1$. By adding a suitable constant we may assume that $f_1$ does not have vanishing radial limits anywhere on $T$.\\
 Now since $E_2$ is $F_\sigma$ it follows $T\setminus E_2$ is of type $G_\delta$ and therefore by {Theorem 4} there exists a non zero bounded analytic function $f_2$ such that $f_2=0$ exactly on $T\setminus E_2$ and such that the radial limits of $f_2$ exist at all points of $T$.
The function $F\equiv f_1f_2$ satisfies the conditions of the theorem hence the proof is complete. 
\end{proof}

\noindent Now recall that by the Factorization Theorem (see p.67, \cite{HO}) every function in $H^1$ can be factored uniquely into the product of three different types of functions. \\
 
 \noindent\textbf{Definition 1.}
A \emph{Singular inner function} is an analytic zero free function $S$ in the unit disc, such that $|S(z)|\leq 1$, $S(0)>0$ and $|S(e^{i\theta})|=1$ almost everywhere on the unit circle. An \emph{outer function} is an analytic function $F$ in the unit disc of the form $$F(z)=\lambda \exp\left[ \frac{1}{2\pi}\int_{-\pi}^{\pi}\frac{e^{i\theta}+z}{e^{i\theta}-z}k(\theta)d\theta\right]$$
where $k$ is a real valued integrable function on the circle and $\lambda$ is a complex number of modulus $1$.\\
 
\noindent The factorization theorem states that every $f\in H^1$ can be factored uniquely as $f=BSF$ where $B$ is a Blaschke product, $S$ a singular inner function and $F$ an outer function. Factoraization Theorem allows us to analyze a function in $H^1$ by looking at each of its factors independently. Our goal now is to solve an important case of the unification problem for one of the three types of functions comprising the factorization, namely Blaschke products. One might then continue the research by examining the problem under the lens of the remaining types of functions in the factorization theorem. Hopefully, by examining the problem with each type of functions one will indeed manage to solve the problem for a general function in $H^\infty$.\vspace{0.12in} It turns out that the particular case of the unification problem we will solve for Blaschke products also extends an important Theorem by Colwell \cite{PE}.\\

Before we proceed we will recall a well know Theorem in regarding the boundary behavior of bounded analytic function and then look at some important results about Blaschke products that turn out to be useful for solving the specific case of the unification problem for Blaschke products.

\noindent The following theorem is well known in analytic function theory.
\begin{theorem}[Fatou's Theorem]
Let $f\in H^\infty$. Then the radial limits of $f$ exist on $T$ except perhaps for a subset $E$ of measure zero.
\end{theorem}
\noindent The following theorem has been proven recently and its proof can be found in \cite{Sp}.

\begin{theorem}
Let E be a set on $T$. Then there exists a Blaschke product which has no radial limits on $E$ but has unrestricted limit at each point of $T\setminus E$ if and only if $E$ is a closed set of measure zero.
\end{theorem}

\noindent The proof of Theorem $8$ uses Fatou's Theorem and some results on the boundary behavior of Blaschke products due to R.D. Berman \cite{BE} and A. Nicolau \cite{NI}.

\noindent The next theorem is due to Berman (\cite{BE}, p. 250).

\begin{theorem}[R.Berman] Let $E$ be a subset of the unit circle of zero Lebesgue measure and of type $F_\sigma$ and $G_\delta$. Then there exist Blaschke products $B_0$ and $B_1$ such that:\\
(i) $B_0$ extends analytically to $T\setminus \overline{E}$ and $\lim_{r\rightarrow 1} B_0(re^{it})=0$ if and only if $e^{it}\in E$;\\
(ii) $\lim_{r\rightarrow 1} B_1(re^{it})=1$ if and only if $e^{it}\in E$.  
\end{theorem}
\noindent Now we formulate a theorem of Nicolau (see \cite{NI}, Proposition on p. 251).
\begin{theorem}[A. Nicolau]
Let $E$ be a subset of the unit circle. Assume that there exist a Blaschke product $B_0$ that extends analytically to $T\setminus \overline{E}$ with $\lim_{r\rightarrow 1}B_0(re^{it})=0$ for $e^{it}\in E$, and an analytic function $f_1$ in the unit ball of $H^\infty$, $f_1\neq 1$, such that $\lim_{r\rightarrow 1}f_1(re^{it})=1$ for $e^{it}\in E$. Then for each analytic function $g$ in the unit ball of $H^\infty$, there exists a Blaschke product $I$ that extends analytically to $T\setminus \overline{E}$, such that $$\lim_{r\rightarrow 1}[I(re^{it})-g(re^{it})]=0 \text{ for } e^{it}\in E$$
\end{theorem}

\noindent The following theorem is a corollary of the necessity part of Rene-Louis Baire's topological theorem for functions in the Baire first class (see \cite{B}, p. 462, Baire's Theorem) which we formulate for functions on the unit circle $T$.

\begin{theorem}[Baire's  Theorem] Let $f$, $f_n$ be defined on $T$ and let $f_n$ be continuous on $T$. If $\lbrace f_n\rbrace$ converges to $f$ at every point of $T$, then $f$ is continuous on a dense subset of $T$.
\end{theorem}

\noindent\textbf{Definition 2.} Let $T$ be a complete metric space and $Z$ a Banach space. A function $f:T\rightarrow Z$ is said to be of the \emph{Baire first class}, if $f$ is the pointwise limit of a sequence of continuous functions.\\

\noindent The following variation of Baire's Theorem also holds (See \cite{S}, p. 986).

\begin{theorem}
Let $T$ and $Z$ be as above and $f:T\rightarrow Z$ be of the Baire first class, then for every closed set $C\subset T$, the function $f\mid C$ has a point of continuity in $C$.
\end{theorem}

\noindent From that, it follows as a corollary that the points of discontinuity of $f$ are nowhere dense in $T$.
\\

\noindent We will use Theorem 8 from above and the corollary of Baire's Theorem 12 mentioned above to prove a special case of the unification problem which extends the main theorem proven by Collwel in \cite{PE}.

\begin{theorem}[Collwel's Theorem] Let $E\subset T$. There exists a Blaschke product $B(z;A)$ for which $B(e^{i\theta})=\lim_{r\rightarrow 1}B(re^{i\theta})$ is well-defined and of modulus one at every point of $T$, satisfying $A'=E$ where $A'$ is the set of accumulation points of the zero set $A$ of our Blaschke product $B(z;A)$, if and only if $E$ is closed and nowhere dense.
\end{theorem}

\noindent Now we are ready to state and prove the following theorem which extends Collwel's theorem and also provides a different and more elegant approach in proving the necessity part of Collwel's Theorem.
\\

\noindent\textbf{Remark 1.}
\noindent Firstly we note that Collwel's Theorem can be used to provide a significant simplification of the proof of the sufficiency part of Theorem 1 found in (\cite{PI}, p.5). The sufficiency part can be stated as the following Theorem:
\begin{theorem}[Lohwater and Piranian] Let $K$ be an $F_\sigma$ set of first category in the unit circle $T$. Then there exists a bounded analytic function $\Phi(z)\in H^\infty$ such that the radial limits of $\Phi(z)$ exists everywhere on $T$ and has unrestricted limits exactly on $T\setminus K$.
\end{theorem}
\begin{proof}
Since $K$ is $F_\sigma$ and of first category, it follows $K=\bigcup_{i=1}^\infty K_i$ where $K_i$ are closed nowhere dense subsets of $C$. Then for each $i$, by Collwel's Theorem there exists a Blaschke product $B(z,A_i)$ that has radial limits of modulus $1$ everywhere on $T$ and has unrestricted limits of modulus $1$ exactly on $T\setminus K_i$.\\
Now since for all $z\in D$ $\mid B(z,A_i)\mid<1$, it follows that $\sum_{i=1}^\infty \frac{\mid B(z,A_i)\mid}{i^2}<\frac{\pi^2}{6}$, thus the infinite sum $\sum_{i=1}^\infty(\frac{B(z,A_i)}{i^2})$ converges uniformly to an analytic function $\Phi(z)$ on $D$. Now since each $B(z,A_i)$ has radial limits everywhere on $T$ and unrestricted limits of modulus $1$ exactly on $T\setminus K_i$, it follows that $\Phi(z)$ has radial limits everywhere on $T$ and unrestricted limits exactly on the $G_\delta$ set $T\setminus K$. The proof is complete.
\end{proof}

\noindent The sufficiency part of Collwel's Theorem in \cite{PE}, that is for each $E\subset T$ that is closed and nowhere dense there exists a Blaschke product $B(z;A)$ with $A'=E$ is based on the following result due to Frostman.

\begin{theorem}[Frostman's condition] Let $A$ be a sequence of points inside the unit disc $D$ and let $B(z,A)$ denote the Blaschke product whose zero set is $A$. Then a necessary and sufficient condition that $\lim_{r\rightarrow 1}B(re^{i\theta^0})=L$ where $|L|=1$ is that \begin{equation}
\sum_A\frac{1-|a|}{|e^{i\theta}-a|}<\infty.
\end{equation}
\end{theorem}

\noindent Now we will use {Frostman's condition} $(1)$ and prove the following:

\begin{prp} Let $B(z,A)$ be a Blaschke product whose radial limits exist everywhere on the unit circle $T$ and with zero set $A=\lbrace a_n \rbrace$. Then the subset of $A'$ in the unit circle in which $B(z,A)$ has radial limits of modulus $1$ is a set that is both $G_{\delta}$ and $F_\sigma$. \end{prp}

\begin{proof} Indeed if we define the following sequence of functions $\lbrace f_n \rbrace$

\begin{align*}
f_n:[0,2&\pi]\rightarrow\mathbb{R}_{>0}\\
&\theta\longmapsto\sum_{k=1}^{n}\frac{1-|a_k|}{|e^{i\theta}-a_k|}.
\end{align*}

\noindent We note that the sequence $\lbrace f_n\rbrace$ is a sequence of continuous functions on $[0,2\pi]$ and the sequence is unbounded exactly at the values of $\theta$  where {Frostman's condition} $(1)$ from above is not satisfied. That is exactly at the points where the Blaschke product $B(z,A)$ fails to have radial limits of modulus $1$. Now, since $f_n$ are continuous functions, the set of points where the sequence $\lbrace f_n\rbrace$ is unbounded, must be of type $G_\delta$. Indeed, for each integer $k \ge 1$, define the open set
\[
E_k = \bigcup_{n=1}^\infty \bigl\{ x \in [0,2\pi] : |f_n(x)| > k \bigr\}.
\]
Observe that a point $x$ belongs to $\bigcap_{k=1}^\infty E_k$ if and only if for every $k$ there exists $n$ with $|f_n(x)| > k$, that is, $\sup_n |f_n(x)| = +\infty$. Therefore, the set where the sequence $\{f_n\}$ is unbounded is exactly $\bigcap_{k=1}^\infty E_k$, which is of type $G_\delta$ as a countable intersection of open sets.
 Therefore by Frostamn's condition It follows that the set of points on $T$ where the radial limits of $B(z,A)$ exist and are of modulus $1$ is of type $F_\sigma$. Now let $r_n=1-\frac{1}{n}$ and consider the following set of continuous functions $$B^*_n(e^{i\theta})\equiv B^*_n(r_ne^{i\theta})=\frac{1}{1-|B_{{r_n}}(e^{i\theta})|}.$$\\

\noindent Then, clearly the set of points $\lbrace e^{i\theta} |\theta \in [0,2\pi]\rbrace$,  where the sequence of functions $B^*_n(e^{i\theta})$ is unbounded, is exactly the set of $\theta\in [0, 2\pi]$ where $\limsup_{r\rightarrow 1}B(re^{i\theta})=1$. But since the radial limits of out Blaschke product $B(z)$ exist everywhere on the unit circle $T$ , this is exactly the set of $\theta\in[0,2\pi]$ where $\lim_{r\rightarrow 1}B(re^{i\theta})=1$. It follows that the subset of $T$ in which the radial limits of $B(z;A)$ exist and are of modulus $1$ is of type $G_\delta$ as well, hence our proof is complete.
\end{proof}

\section{A case of the unification problem as an analogue of Kolesnikov's Theorem for Blaschke products} 

\noindent Now we will prove another theorem regarding the boundary behavior of Blaschke products. To prove the theorem we use the following theorem found in \cite{NI}.

\begin{theorem} Let $E$ be a subset of the unit circle of zero Lebesgue measure and of type $F_\sigma$ and $G_\delta$. Let $\phi$ be a function defined on $E$ with $\sup\lbrace |\phi(e^{it})|:e^{it}\in E\rbrace\leq 1$ and such that for each open set $U$ in the complex plane, $\phi^{-1}(U)$ is of type $F_\sigma$ and $G_\delta$. Then there exists a Blaschke product $I$ extending analytically on $T\setminus \overline{E}$ such that $$\lim_{r\rightarrow 1} I(re^{it})=\phi(e^{it})\text{ for }e^{it}\in E.$$
\end{theorem}

\noindent The theorem we want to prove is the following:
\begin{theorem} Let $E_1\subset E_2$ be subsets of the unit circle $T\equiv\lbrace z\mid |z|=1 \rbrace$. Moreover, assume that $E_2$ is of measure zero. Then there exists a Blaschke product $B$ such that the radial limits of $B$ fail to exist precisely at $E_1$ and it has unrestricted limits precisely at $T\setminus E_2$ if and only if $E_2$ is closed and $E_1$ is of type $G_{\delta\sigma}$ and of measure zero.
\end{theorem}

\begin{proof}\textbf{(Necessity.)}
The proof that $E_2$ is closed follows since any Blaschke product has unrestricted limits only on the set complimentary to the accumulation of its zeros. This implies that the set $E_2$ is the closed set in which the zeros of our Blaschke product accumulate. The fact that $E_1$ is $G_{\delta\sigma}$ is well known and as noted earlier its proof may be found in \cite{AL}. \\
\textbf{(Sufficiency.)}
Since $E_1$ is a $G_{\delta\sigma}$ of measure zero we know by {Kolesnikov's Theorem} that there exists a bounded analytic function $f_1$ with no radial limits exactly on $ E_1$. Moreover, since $E_2$ is a closed zero measure set we can apply {Theorem 8} to find a Blaschke product $B_2$ such that the radial limits of $B_2$ are equal to $\frac{1}{2}$ for all $e^{it}\in E_2$. Moreover, since $E_2$ is closed the conditions of Nicolau's Theorem are satisfied as well. It follows that for every bounded analytic function $g$ there exists a Blaschke product $I$ mimicking its boundary behavior on the set $E_2$. Now the product function $g(z)\equiv f_1(z)\cdot B_2(z)$ has the desired boundary behavior restricted on $E_2$; that is $g$ has no radial limits exactly on $E_1$ and has radial but not unrestricted limits on $E_2\setminus E_1$ . Thus, by Nicolau's Theorem we can find a Blaschke product $I$ that extends analytically on $T\setminus E_2$ and such that $$\lim_{r\rightarrow 1}[ I(re^{it})-g(re^{it}]=0$$
Then the Blaschke product $I$ has the desired boundary behavior and the proof is complete.\\

\noindent\textbf{Remark 3.}
\noindent Recall that {Kolesnikov's Theorem} states that for any set $E$ of type $G_{\delta\sigma}$ on the unit circle there exists a bounded analytic function which fails to have radial limits precisely on that set. Thus Theorem 18 can be thought of as an analogue of {Kolesnikov's Theorem} for Blaschke products, but with the extra requirement that $\overline{E}$ is of measure zero.

\end{proof}

\end{document}